\documentclass[12pt]{article}

\usepackage{amsmath,amsthm,amsfonts,amssymb}

\newtheorem{theorem}{Theorem}[section]

\newtheorem{proposition}[theorem]{Proposition}

\newtheorem{remark}[theorem]{Remark}

\def\cF{\mathcal{F}}

\def\bR{\mathbb{R}}

\newcommand*\samethanks[1][\value{footnote}]{\footnotemark[#1]}
\begin{document}

\title{SPDEs with rough noise in space:\\ H\"older continuity of the solution}

\author{Raluca M. Balan\footnote{Department of Mathematics and Statistics, University of Ottawa,
585 King Edward Avenue, Ottawa, ON, K1N 6N5, Canada. E-mail
address: rbalan@uottawa.ca. Research supported by a
grant from the Natural Sciences and Engineering Research Council
of Canada.}\and
Maria Jolis\thanks{Departament de Matem\`atiques,
Universitat Aut\`{o}noma de Barcelona, 08193 Bellaterra (Barcelona), Catalonia, Spain.
E-mail addresses: mjolis@mat.uab.cat, quer@mat.uab.cat. Research supported by grants
MCI-FEDER MTM2012-33937 and SGR 2014-SGR-422. Corresponding author: L. Quer-Sardanyons.}
\and
Llu\'{i}s Quer-Sardanyons  \samethanks 
}
\date{Jannuary 15, 2015}
\maketitle

\begin{abstract}
\noindent We consider the stochastic wave and heat equations with affine multiplicative Gaussian noise which is white in time and behaves in 
space like the fractional Brownian motion with index $H \in (\frac14,\frac12)$. The existence and uniqueness of the solution to 
these equations has been proved recently by the authors.
In the present note we show that these solutions have modifications which are H\"older continuous in space of order smaller than $H$, and H\"older continuous in time of order smaller than $\gamma$, where $\gamma=H$ for the wave equation and $\gamma=H/2$ for the heat equation.
\end{abstract}


\vspace{3mm}

\section{Introduction}

In this article, we consider the stochastic wave equation:
\begin{equation}
\left\{\begin{array}{rcl}
\displaystyle \frac{\partial^2 u}{\partial t^2}(t,x) & = & \displaystyle \frac{\partial^2 u}{\partial x^2}(t,x)+\sigma(u(t,x))\dot{X}(t,x), \quad t\in [0,T], \ x \in \bR \\[2ex]
\displaystyle u(0,x) & = & u_0(x), \\[1ex]
\displaystyle \frac{\partial u}{\partial t}(0,x) & = & v_0(x),
\end{array}\right. \label{wave} \tag{SWE}
\end{equation}
and the stochastic heat equation:
\begin{equation}
\left\{\begin{array}{rcl}
\displaystyle \frac{\partial u}{\partial t}(t,x) & = &\displaystyle  \frac{1}{2}\, \frac{\partial^2 u}{\partial x^2}(t,x) + \sigma(u(t,x))\dot{X}(t,x), \quad t\in [0,T], \ x \in \bR \\[2ex]
\displaystyle u(0,x) & = & u_0(x)
\end{array}\right. \label{heat} \tag{SHE}
\end{equation}
where $\sigma(x)=ax+b$ with $a,b \in \bR$, and $u_0$ and $v_0$ are uniformly H\"older continuous 
of order $H$. We assume that the noise $\dot{X}$ is white in time and behaves in space like a fractional Brownian motion (fBm)
with index $H \in (\frac14,\frac12)$. More precisely, $\dot{X}$ is the formal derivative of a zero-mean 
Gaussian process $X=\{X(\varphi); \varphi \in C_0^{\infty}((0,\infty) \times \bR)\}$ with 
covariance given by:
\begin{equation}
\label{cov-noise}
E[X(\varphi)X(\psi)]=c_H\int_{0}^{\infty}\int_{\bR} \cF \varphi(t,\cdot)(\xi) \overline{\cF 
\psi(t,\cdot)(\xi)}|\xi|^{1-2H} d\xi dt,
\end{equation}
for any $\varphi,\psi \in C_0^{\infty}((0,\infty) \times \bR)$, where $c_H=\Gamma(2H+1) \sin (\pi H)/(2\pi)$. Here $C_0^{\infty}(\bR_{+} \times \bR)$ denotes the space of infinitely differentiable functions on $\bR_+ \times \bR$ with compact support.

Since the Fourier transform of the measure $\mu(d\xi)=c_H|\xi|^{1-2H}d\xi$ is not given by a locally 
integrable function, the study of equations with this kind of noise does not fall under the general theory 
of SPDEs with colored noise initiated by \cite{dalang99} and \cite{PZ}. Instead, the family $X$ can be seen as 
a random stationary distribution (see \cite{ito54,yaglom57}), which allows to define stochastic 
integrals with respect to this type of noise following the ideas of Basse-O'Connor {\it{et al.}} 
\cite{BGP12}. See \cite[Sec. 2]{BJQ} for a detailed description of the noise in this setting 
as well as the construction of the corresponding stochastic integrals. 

In \cite[Thm. 1.1]{BJQ}, we proved the existence of a unique mild solution to equations 
\eqref{wave} and \eqref{heat} 
in the space $\mathcal{X}_p$, for any fixed $p\geq 2$, where the latter is defined as 
the space of $L^2(\Omega)$-continuous and
adapted processes $u=\{u(t,x); t \in [0,T], x\in \bR \}$ satisfying
$$\sup_{(t,x) \in [0,T] \times \bR}E\big[|u(t,x)|^p\big]<\infty$$
and
\begin{equation}
 \sup_{(t,x) \in [0,T] \times \bR} 
 \int_0^t \int_{\bR^2}G_{t-s}^{2}(x-y)\dfrac{\Big(E\big[|u(s,y)-u(s,z)|^p\big]\Big)^{2/p}}{|y-z|^{2-2H}}\,dy\,dz\,ds <\infty.
 \label{eq:48}
\end{equation}
Here, $G_t(x)$ denotes the fundamental solution of the wave (respectively heat) equation, that is
$$G_t(x)=\frac{1}{2}1_{\{|x|<t\}} \quad \mbox{for the wave equation},$$
$$G_t(x)=\frac{1}{(2\pi t)^{1/2}} \exp\left(-\frac{|x|^2}{2t}\right) \quad \mbox{for the heat equation.}$$
The method used in \cite{BJQ} to prove existence and uniqueness of solution is based on a 
Picard iteration scheme. We point out that the term (\ref{eq:48}) pops up in a quite natural way, 
for we used some harmonic analysis techniques related to fractional Sobolev spaces. 

We recall that a random field $u=\{u(t,x); t \in [0,T],
x\in \bR \}$ is a solution of \eqref{wave} (respectively \eqref{heat})
if $u$ is predictable and, for any $(t,x) \in [0,T] \in \bR$,
\begin{equation*}
\label{int-eq}
u(t,x)=w(t,x)+\int_0^t\int_{\bR} G_{t-s}(x-y)\,\sigma(u(s,y))\,X(ds,dy) \quad {\rm a.s.}
\end{equation*}
where the stochastic integral is interpreted in the sense explained in 
\cite[Sec. 2]{BJQ}. In the above expression, $w(t,x)$ denotes the solution of the 
corresponding homogeneous equation; see the beginning of Section \ref{sec:proof}
for the precise expression of $w(t,x)$ for wave and heat equations.

The goal of the present note is to show that the solutions of \eqref{wave} and \eqref{heat} 
have H\"older continuous modifications in space and time. 
More precisely, we will prove the following result.

\begin{theorem}
\label{main-th}
Let $u=\{u(t,x);t \in [0,T],x \in \bR\}$ be the solution to equation \eqref{wave}, respectively equation \eqref{heat}. There exists $h_0\in (0,1)$ such that, for all $|h|\leq h_0$ and for any $p \geq 2$, we have:
$$\sup_{(t,x) \in [0,T] \times \bR} \Big( E\big[|u(t,x+h)-u(t,x)|^p]\Big)^{\frac1p} \leq C_p\, |h|^H$$
and
$$\sup_{(t,x) \in [0,T \wedge (T-h)] \times \bR} \Big(E \big[|u(t+h,x)-u(t,x)|^p\big]\Big)^{\frac1p}
  \leq C_p\, |h|^{\gamma},$$
where $C_p>0$ is a constant depending on $p$, 
and $\gamma=H$ for the wave equation and $\gamma=\frac H2$ for the heat
equation.
 { Therefore, the random field $u$ has a modification that has  
 $(\gamma',\,H')$-H\"older-continuous sample paths,
  for any $\gamma'<\gamma$
 and any $H'<H$.}
\end{theorem}

We note that the stochastic heat equation with the same noise
$\dot{X}$ as above has been thoroughly studied in the recent
preprint \cite{HHLNT}, in the case of a Lipschitz function $\sigma$
with Lipschitz derivative and such that $\sigma(0)=0$. In this
article, the authors have obtained an exponential upper bound for
the $p$-th moment of the solution and have shown that that this
solution has a H\"older continuous modification of order
$(\frac{H}{2}-\varepsilon,H-\varepsilon)$ for any $\varepsilon>0$
(see Theorem 4.31 of \cite{HHLNT}). Moreover, in the case when
$\sigma(x)=x$, the authors of \cite{HHLNT} have obtained a
Feynman-Kac representation for the moments of the solution to the
heat equation, which was used to show that these moments grow
exponentially in time (Theorems 5.7 and 5.8 of \cite{HHLNT}). These
impressive investigations were continued in the recent preprint
\cite{HLN} in which the authors computed the exact Lyapunov
exponents and the lower and upper growth indices of the solution of
the heat equation with noise $\dot{X}$, in the case $\sigma(x)=x$.

One of the key steps which allow the authors of \cite{HHLNT} to obtain a solution 
with $\sigma$ satisfying the above-mentioned conditions is based on a 
localization argument which is tied to the parabolic nature of the heat equation. 
In this sense, an important characteristic of our method in \cite{BJQ} 
is that we can deal with both heat and wave equations at the same time.

In the next Section we proceed to prove Theorem \ref{main-th}.

%%%%%%%%%%%%%%%%%%%%%%%%%%%%%%%%%%%%%%%%%%%%%%%%%

\section{Proof of Theorem \ref{main-th}}
\label{sec:proof}

The proof of Theorem \ref{main-th} will follow from a careful analysis of the $p$-th moments of the increments of the Picard iteration sequence.

Along this note $w(t,x)$ will denote 
the solution of the homogeneous wave equation with the same initial
conditions as \eqref{wave} (respectively of the homogeneous heat
equation), that is :
$$w(t,x)=\frac{1}{2}\int_{x-t}^{x+t}v_0(y)dy+\frac{1}{2}
\Big(u_0(x+t)+u_0(x-t)\Big) \quad \mbox{for the wave equation and}$$
$$w(t,x)= \int_{\bR}G_t(x-y)u_0(y)dy \quad \mbox{for the heat equation}.\phantom{xxxxxxxxxxxxxxxxxxxxx}$$

Let $(u^n)_{n \geq0}$ be the Picard iteration scheme defined by: $u^0(t,x)=w(t,x)$ and
$$u^{n+1}(t,x)=w(t,x)+\int_0^t \int_{\bR}G_{t-s}(x-y)\sigma(u^n(s,y))X(ds,dy), \quad n \geq 0.$$

First, in Section 3.2 of \cite{BJQ} we proved the following result.

\begin{theorem}
\label{Picard-well-def}
%Let $\sigma$ be an arbitrary Lipschitz function.
Let $p \geq 2$ be fixed. Then, for any $n \geq 0$,
\begin{equation}
\left.\begin{array}{rcl}
& & \displaystyle u^n(t,x) \ \mbox{is well-defined for any} \ (t,x) \in [0,T] \times \bR,  \\[2ex]
& & \displaystyle \sup_{(t,x) \in [0,T] \times \bR}E|u^n(t,x)|^p<\infty, \quad \mbox{and} \\[1ex]
& & \displaystyle \sup_{(t,x) \in [0,T] \times \bR} \int_0^t \int_{\bR^2}G_{t-s}^2(x-y)\frac{\Big(E|u^n(s,y)-u^n(s,z)|^p\Big)^{2/p}}{|y-z|^{2-2H}}\,dy \,dz \,ds<\infty
\end{array}\right\} \label{propertyP} \tag{P}
\end{equation}
and, for any $h\in \bR$ with $|h|<1$,
\begin{equation}
\left.\begin{array}{rcl}
& & \displaystyle \sup_{(t,x) \in [0,T] \times \bR}E|u^n(t,x+h)-u^n(t,x)|^2 \leq C_n |h|^{2H} \\
& & \displaystyle \sup_{(t,x) \in [0,T \wedge (T-h)] \times
\bR}E|u^n(t+h,x)-u^n(t,x)|^2 \leq C_n |h|^{\beta},
\end{array} \right\} \label{propertyQ} \tag{Q}
\end{equation}
where $\beta=2H$ for the wave equation, and $\beta=H$ for the heat
equation. Here $C_n$ is a constant which depends on $n$ (and also on
$H,T,\sigma,u_0$ and $v_0$).
\end{theorem}

%Assume that $\sigma$ is affine.
 On the other hand, an immediate consequence of \cite[Thm 3.9]{BJQ} is that,
for all $p\geq 2$,
\begin{equation}
 \sup_{n\geq 0} \sup_{(t,x) \in [0,T] \times \bR}E|u^n(t,x)|^p<\infty.
 \label{eq:0}
\end{equation}

Now, we aim to improve property (Q) above in the following sense.

\begin{proposition}\label{prop:1}
Let $p \geq 2$  and $h_0\in (0,1)$. Then, for any $n\geq 0$,
 \begin{equation}
\left.\begin{array}{rcl}
& & \displaystyle \sup_{(t,x) \in [0,T] \times \bR}\big( E [|u^n(t,x+h)-u^n(t,x)|^p]\big)^{\frac1p}
\leq C_n |h|^H \\
& & \displaystyle \sup_{(t,x) \in [0,T \wedge (T-h)] \times
\bR} \big( E[|u^n(t+h,x)-u^n(t,x)|^p]\big)^{\frac1p}  \leq C_n |h|^{\gamma},
\end{array} \right\} \label{propertyQp} \tag{Q'}
\end{equation}
for all $|h|\leq h_0$, where $\gamma=H$ for the wave equation, and $\gamma=\frac H2$ for the heat
equation, and the constant $C_n$ satisfies
\[
 C_n\leq C\, \big(c(h_0)+\bar{c}(h_0) C_{n-1}\big).
\]
The functions $c,\bar{c}:\bR\rightarrow \bR$ are non-negative and
$\lim_{h_0\rightarrow 0} \bar{c}(h_0)=0$. By definition, $C_{-1}=0$.
\end{proposition}

\begin{proof}
 We split the proof in four steps. We will only develop in detail the computations which are 
 relevant to attain our main objective, so that the reader will be directed to \cite{BJQ} 
 for similar arguments or computations. 

 \medskip

 \noindent {\it Step 1.} The case $n=0$ follows from the first part of the proof of
 \cite[Thm 3.7]{BJQ}. More precisely, for the wave equation we proved that
 \begin{equation}
  \sup_{(t,x) \in [0,T] \times \bR} |w(t,x+h)-w(t,x)|\leq C\, |h|^H
  \label{eq:1}
 \end{equation}
 and
 \[
  \sup_{(t,x) \in [0,T \wedge (T-h)] \times \bR} |w(t+h,x)-w(t,x)|\leq C\, \big(|h|^H +
  |h|\big) \leq C\, \big(1 + h_0^{1-H}\big) |h|^H.
 \]

 On the other hand, for the heat equation we obtained the same estimate (\ref{eq:1}) for the
 space increments, and
 \[
  \sup_{(t,x) \in [0,T \wedge (T-h)] \times \bR} |w(t+h,x)-w(t,x)|\leq C\, |h|^{\frac H2}.
 \]
Thus, we obtain condition (Q') with $C_0:= C\, \big(1 + h_0^{1-H}\big) \leq C$.

\medskip

\noindent {\it Step 2.} Induction step.
We first consider the space increments of $u^{n+1}$. We have, thanks to 
a [Burkholder-Davis-Gundy]-type inequality for stochastic integrals with respect to our fractional noise $X$
(see \cite[Thm. 2.9]{BJQ}), 
$$\big( E[|u^{n+1}(t,x+h)-u^{n+1}(t,x)|^p]\big)^{\frac1p} \leq C(I_0+I_1+I_2),$$
where $I_0=|w(t,x+h)-w(t,x)|$,
\begin{align*}
I_1 &= \left( E \left| \int_0^t \int_{\bR^2} |G_{t-s}(x+h-y)-G_{t-s}(x-y)|^2
\dfrac{|\sigma(u^n(s,y))-\sigma(u^n(s,z))|^2}{|y-z|^{2-2H}}\,dy \,dz\,ds \right|^{\frac p2}\right)^{\frac1p}\\
I_2 &= \left( E \left|\int_0^t \int_{\bR^2} \frac{|\sigma(u^n(s,z))|^2}{|y-z|^{2-2H}}\,
|(G_{t-s}(x+h-y)-G_{t-s}(x-y)) \right.\right.\\
& \qquad \qquad \qquad \qquad \qquad - (G_{t-s}(x+h-z)-G_{t-s}(x-z))|^2\,dy \,dz \,ds\Big|^{\frac p2}
\Bigg)^{\frac1p}.
\end{align*}
We have already proved that $I_0 \leq C_0 |h|^H$. Let us treat
$I_1$. By Minkowski's inequality and using that $\sigma$ is Lipschitz, we have
\begin{align*}
 I_1^2 & \leq C  \int_0^t \int_{\bR} |G_{t-s}(x+h-y)-G_{t-s}(x-y)|^2 \\
 & \qquad \qquad \times
 \left(\int_{\bR}\dfrac{ \big( E [ |u^n(s,y+z)-u^n(s,y)|^p] \big)^{\frac 2p}}{|z|^{2-2H}}\,dz\right)
 \,dy\,ds \\
 & =: C \,(I_1'+I_1''),
\end{align*}
where $I_1'$ and $I_1''$ denote the integrals corresponding to the
regions $\{|z|>h_0\}$, respectively $\{|z| \leq h_0\}$, in the $dz$
integral. By (\ref{eq:0}) and taking into account that
$\int_{|z|>h_0} |z|^{2H-2}dz= Ch_0^{2H-1}$, we have (as in page 22
of \cite{BJQ})
\begin{align*}
 I'_1 & \leq C\, h_0^{2H-1} \int_0^t \int_{\bR} |G_{t-s}(x+h-y)-G_{t-s}(x-y)|^2  \,dy\,ds\\
 & \leq C\, h_0^{2H-1} |h| = C\, h_0^{2H-1} |h|^{1-2H} |h|^{2H} \leq C\, |h|^{2H}.
\end{align*}
On the other hand, by the induction hypothesis and using that
$\int_{|z|\leq h_0} |z|^{4H-2}dz=C\,h_0^{4H-1}$, it holds
\begin{align*}
 I''_1 &\leq C \, C_n^2 \, h_0^{4H-1} |h| \leq C \, C_n^2 \, h_0^{2H} |h|^{2H}.
\end{align*}
Thus $I_1\leq C\, (1+h_0^H C_n) |h|^H$.

In order to deal with $I_2$, we apply again Minkowski's inequality, the linear growth on $\sigma$ and (\ref{eq:0}), and we argue as in the last part
of page 22 in \cite{BJQ}:
\begin{align*}
 I_2^2 & \leq C\, \int_0^t \int_{\bR}(1-\cos(h|\xi|)) \, |\cF G_{t-s}(\xi)|^2 \,|\xi|^{1-2H}\,d\xi \,ds \\
 & \leq C\, |h|^{2H},
\end{align*}
which implies that $I_2\leq C\, |h|^H$. Hence, putting together the
estimates for $I_0$, $I_1$ and $I_2$, we have proved that
\begin{equation}
 \sup_{(t,x) \in [0,T] \times \bR} \big( E[|u^{n+1}(t,x+h)-u^{n+1}(t,x)|^p]\big)^{\frac1p} \leq
 C\, \big(C_0 + h_0^H C_n\big) |h|^H.
 \label{eq:3}
\end{equation}

\medskip

\noindent {\it Step 3.} Let us now consider the time increments. We
consider the case $h\ge 0$, being similar the case $h<0$. We have
that
$$\big(E[|u^{n+1}(t+h,x)-u^{n+1}(t,x)|^p]\big)^{\frac1p}  \leq C (J_0+J_1+J_2),$$
where $J_0=|w(t+h,x)-w(t,x)|$ ,
\[
J_1 = \left( E \left| \int_{t}^{t+h}\int_{\bR^2}\dfrac{|G_{t+h-s}(x-y)\sigma(u^n(s,y))-G_{t+h-s}(x-z)
\sigma(u^n(s,z))|^2}{|y-z|^{2-2H}}\,dy\,dz\,ds\right|^{\frac p2} \right)^{\frac 1p},
\]
\begin{align*}
J_2& =  \left( E \left|\int_0^t \int_{\bR^2}|(G_{t+h-s}(x-y)-G_{t-s}(x-y))\sigma(u^n(s,y)) \right.\right.\\
&  \quad \quad \qquad
- (G_{t+h-s}(x-z)-G_{t-s}(x-z))\sigma(u^n(s,z))|^2|y-z|^{2H-2}\,dy\,dz \,ds\Big|^{\frac p2} \Bigg)^{\frac 1p}.
\end{align*}
We have already seen that $J_0 \leq C_0 |h|^\gamma$.

As far as $J_1$ is concerned, we apply Minkowski's inequality and we add and subtract the term
$G_{t+h-s}(x-y)\sigma(u^n(s,z))$. We obtain
\begin{align*}
 J_1^2 & \leq \int_{t}^{t+h}\int_{\bR^2}\dfrac{ \big( E [|G_{t+h-s}(x-y)\sigma(u^n(s,y))-G_{t+h-s}(x-z)
\sigma(u^n(s,z))|^p]\big)^{\frac2p}}{|y-z|^{2-2H}}\,dy\,dz\,ds\\
& \leq C\, (J_{11}+J_{12}),
\end{align*}
where
\[
J_{11} = \int_t^{t+h} \int_{\bR^2} G_{t+h-s}^2(x-y) \dfrac{\big(
E[|\sigma(u^n(s,y))-\sigma(u^n(s,z))|^p] \big)^{\frac2p}
}{|y-z|^{2-2H}}\,dz\,dy\,ds,
\]
\[
J_{12} = \int_{t}^{t+h} \int_{\bR^2}\big(E[|\sigma(u^n(s,z))|^p]\big)^{\frac2p} \,\dfrac{|G_{t+h-s}(x-y) -G_{t+h-s}(x-z)|^2}{ |y-z|^{2-2H}}\,dy\,dz \,ds.
\]
First, we have
\begin{align*}
 J_{11} & \leq C\,
 \int_t^{t+h} \int_{\bR} G_{t+h-s}^2(x-y) \left(\int_{\bR}\dfrac{\big( E[|u_n(s,y+z)-u_n(s,z)|^p]\big)^{\frac2p}}{|z|^{2-2H}}\,dz\right)\,dy\,ds \\
 & = C\, (J_{11}'+J_{11}''),
\end{align*}
where the latter are defined by splitting the $dz$ integrals into
two integrals corresponding to the regions $\{|z|>h_0\}$ and,
$\{|z|\leq h_0\}$, respectively. By (\ref{eq:0}) and using that
$\int_{|z|>h_0} |z|^{2H-2}dz= C\,h_0^{2H-1}$, we obtain (as in page
23 of \cite{BJQ})
\begin{align*}
 J_{11}'& \leq C\, h_0^{2H-1} \int_0^h \int_{\bR}G_s^2(s,y)\,dy \,ds \\
 & \leq C\, h_0^{2H-1} |h|^{\frac{\gamma}{H}}= C\, h_0^{2H-1+(\frac1H-2)\gamma} |h|^{2\gamma},
\end{align*}
where we recall that $\gamma=H$ for the wave equation and $\gamma=\frac H2$ for the heat
equation.

On the other hand, by the induction hypothesis, and using that
$\int_{|z|\leq h_0} |z|^{4H-2}dz=C\,h_0^{4H-1}$, and  that
$\frac{\gamma}{H}-2\gamma>0$ for any $\gamma>0$, we have
\[
 J_{11}''\leq C\, h_0^{4H-1}\, C_n^2 \, |h|^{\frac{\gamma}{H}} \leq C\, C_n^2 \, h_0^{4H-1+\frac{\gamma}{H}-2\gamma} |h|^{2\gamma}.
\]
Observe that the quantity $4H-1 + \frac{\gamma}{H}-2\gamma$ is always positive.

Let us now deal with $J_{12}$. Indeed, by (\ref{eq:0}) and \cite[Prop. 2.8]{BJQ} we have 
\[
 J_{12}\leq C\, \int_{0}^{h} \int_{\bR} |\cF G_{r}(\xi)|^2 \,|\xi|^{1-2H}\, d\xi dr.
\]
By Lemma 3.1 in \cite{BJQ}, the last integral is equal to $C h^{2H+1}$ for the wave equation, and $C h^{H}$ for the heat equation.
Hence,
\[
 J_{12}\leq C\, \big(1+h_0) |h|^{2\gamma}.
\]
Thus, we have proved that
\[
 J_1\leq C\, \big(1 + h_0^\frac12 + h_0^{H-\frac12+(\frac{1}{2H}-1)\gamma} + C_n\, h_0^{2H-\frac12+(\frac{1}{2H}-1)\gamma} \big) |h|^\gamma.
\]

Now we treat the term $J_2$. Arguing as in page 23 of \cite{BJQ} and applying Minkowski's inequality, we can infer that
$J_2^2\leq C\, (J_{21}+J_{22})$, where
\begin{align*}
J_{21} & =  \int_0^t \int_{\bR^2}|G_{t+h-s}(x-y)-G_{t-s}(x-y)|^2\\
&  \quad \quad \qquad
\big(E[|\sigma(u^n(s,y))-\sigma(u^n(s,z))|^p]\big)^{\frac2p}
\,|y-z|^{2H-2}\,dy \,dz \,ds\\
& =  \int_0^t \int_{\bR^2}|G_{t+h-s}(x-y)-G_{t-s}(x-y)|^2\\
&  \quad \quad \qquad
\big(E[|\sigma(u^n(s,y))-\sigma(u^n(s,y+z))|^p]\big)^{\frac2p}
\,|z|^{2H-2}\,dy \,dz \,ds,
\end{align*}
\begin{align*}
J_{22} &= \int_0^t \int_{\bR^2} \big(E[|\sigma(u^n(s,z))|^p]\big)^{\frac2p} \, |(G_{t+h-s}(x-y)-G_{t-s}(x-y))-\\
&  \quad \quad \qquad (G_{t+h-s}(x-z)-G_{t-s}(x-z))|^2 \,|y-z|^{2H-2}\,dy\,dz\,ds.
\end{align*}
Similarly as before, we have
$J_{21} \leq C\,(J_{21}'+J_{21}'')$, where $J_{21}'$ and $J_{21}''$ are integrals corresponding to the regions $\{|z|>h_0\}$, respectively $\{|z| \leq h_0\}$.
Then,
\begin{align*}
 J_{21}' & \leq C\, h_0^{2H-1} \int_0^t \int_{\bR}|G_{t+h-s}(x-y)-G_{t-s}(x-y)|^2 dy\,ds \\
 & \leq C\, h_0^{2H-1} |h|^{\frac{\gamma}{H}} \leq C\, h_0^{2H-1+ (\frac1H-2)\gamma} |h|^{2\gamma}.
\end{align*}
On the other hand, by Lipschitz condition and the induction
hypothesis, we have
\[
 J_{21}''\leq C\, h_0^{4H-1}\, C_n^ 2\, |h|^{\frac{\gamma}{H}}\leq C\, C_n^2 \, h_0^{4H-1+\frac{\gamma}{H}-2\gamma} |h|^{2\gamma}.
\]
Finally, using that $E[|\sigma(u^n(s,z))|^p]$ is uniformly bounded
  on $s$, $z$ and $n$, and  Proposition 2.8 of \cite{BJQ},
\begin{align*}
 J_{22}&\leq C\, \int_0^t \int_{\bR}|\cF G_{t+h-s}(\xi)-\cF G_{t-s}(\xi)|^2 \, |\xi|^{1-2H}\,d\xi\,ds\\
 & \leq C\, |h|^{2\gamma}.
\end{align*}
Hence, we have obtained that
\[
 J_2\leq C\, \big(1 + h_0^{H-\frac12+(\frac{1}{2H}-1)\gamma} + C_n\, h_0^{2H-\frac12+(\frac{1}{2H}-1)\gamma} \big) |h|^\gamma.
\]

Putting together the bounds for $J_0$, $J_1$ and $J_2$, we get
\begin{align}
 & \sup_{(t,x) \in [0,T \wedge (T-h)] \times \bR} \big(E[|u^{n+1}(t+h,x)-u^{n+1}(t,x)|^p]\big)^{\frac1p} \nonumber \\
 & \qquad \qquad \leq C\, \big(1 + C_0 + h_0^\frac12 + h_0^{H-\frac12+(\frac{1}{2H}-1)\gamma} + C_n\, h_0^{2H-\frac12+(\frac{1}{2H}-1)\gamma} \big) |h|^\gamma.
 \label{eq:2}
\end{align}

\medskip

\noindent {\it Step 4.} Finally, by estimates (\ref{eq:3}) and (\ref{eq:2})
we have property (Q') with a constant $C_{n+1}:= C\, \big(c(h_0)+\bar{c}(h_0) C_n\big)$,
where
\[
 c(h_0)= 1 + C_0 + h_0^\frac12 + h_0^{H-\frac12+(\frac{1}{2H}-1)\gamma}
\]
and
\[
 \bar{c}(h_0) = h_0^{2H-\frac12+(\frac{1}{2H}-1)\gamma}.
\]
Observe that for any $\gamma>0$   it holds that
$2H-\frac12+(\frac{1}{2H}-1)\gamma>0$. Hence $\lim_{h_0\rightarrow
0} \bar{c}(h_0)=0$. This concludes the proof.
\end{proof}

\medskip

Now, we are in position to prove our main result.

\medskip

\noindent {\em Proof of Theorem \ref{main-th}}.
Let us first prove that, choosing a small enough $h_0\in (0,1)$, the sequence of constants $(C_n)_{n\geq 0}$ in property (Q') is bounded.
Indeed, using the recursion for the constant $C_{n+1}$, one easily verifies that,
for all $n\geq 0$,
\begin{align*}
 C_{n+1} & = C c(h_0)\big(1+C \bar{c}(h_0) + [C \bar{c}(h_0)]^2  +\cdots +
 [C \bar{c}(h_0)]^n\big) + [C \bar{c}(h_0)]^{n+1} C_0 \\
 & \leq \max(C c(h_0),C_0) \sum_{k=0}^{n+1} [C \bar{c}(h_0)]^k,
\end{align*}
where we recall that $C_0=C\, \big(1+h_0^{1-H}\big)$. Thus, choosing $h_0$
small enough such that $C \bar{c}(h_0)<1$, we get that $\sup_{n\geq 1} C_n<+\infty$.

Hence, we have obtained the validity of property (Q') with the constant $C_n$ replaced by a constant $C$,
which does not depend on $n$. At this point, taking limits as $n$ tends to infinity in (Q'), 
one gets the first part of the statement,
since we already know that $u^n$ converges to $u$ in $L^2(\Omega)$, uniformly in time and space.

The second part of the statement follows from a $d$-dimensional
parameter version of Kolmogorov criterion of continuity; see, for
instance,  Theorem 1.4.1 of \cite{ku}. The proof is complete. \qed

\medskip

\begin{remark}
{\rm An important consequence of Theorem \ref{main-th} is that, in
fact, the solutions of our SPDEs  belong to a smaller space than the
space $\mathcal{X}_p$ defined in the Introduction (see also \cite[Def. 3.6]{BJQ}). 
Precisely, for any $p\geq 2$, they belong to the space of adapted random fields 
$\{u(t,x); t\in [0,T], x\in \bR\}$ satisfying the following three conditions:
$$\sup_{(t,x) \in [0,T] \times \bR}E[|u(t,x)|^p]<\infty,$$
$$\sup_{(t,x) \in [0,T] \times \bR} \big( E|u(t,x+h)-u(t,x)|^p\big)^{\frac1p} \leq C_p\, |h|^H$$
and
$$\sup_{(t,x) \in [0,T \wedge (T-h)] \times \bR} \big(E|u(t+h,x)-u(t,x)|^p\big)^{\frac1p}
  \leq C_p\, |h|^{\gamma}.$$
Indeed, it is easy to see, using the usual argument of splitting the
$dz$ integral, that the above conditions imply that
\[
 \sup_{(t,x) \in [0,T] \times \bR} \int_0^t \int_{\bR^2}G_{t-s}^2(x-y)\frac{\Big(E|u(s,y)-u(s,z)|^p\Big)^{2/p}}{|y-z|^{2-2H}}\,dy \,dz \,ds<\infty.
\]
On the other hand, the processes belonging to the intersection for
all $p\ge 2$ of these spaces have versions with
$(\gamma',H')$-H\" older continuous paths for any $\gamma'<\gamma$ and
any $H'<H$.

%We denote also by $u$ the H\"older-continuous modification of the solution given by Theorem \ref{main-th}.
%We have that,
%for small enough $h$ and $\varepsilon>0$,
%\[
% \sup_{(t,x)\in [0,T]\times \bR} |u(t,x+h)-u(t,x)|\leq C\, |h|^{H-\varepsilon} \quad \mathbb{P}\text{-a.s.}
%\]
%This condition implies that (using the usual argument of splitting the $dz$ integral),
%\[
% \sup_{(t,x) \in [0,T] \times \bR} \int_0^t \int_{\bR^2}G_{t-s}^2(x-y)\frac{\Big(E|u(s,y)-u(s,z)|^p\Big)^{2/p}}{|y-z|^{2-2H}}\,dy \,dz \,ds<\infty.
%\]
%%Hence, we have indeed proved that the solution lives in a smaller space than $\mathcal{X}$:
%the space of adapted random fields $u$ with uniform space-time H\"older-continuous paths with indices $H-\varepsilon$ and
%$H-\varepsilon$, respectively, and satisfying, for all $p\geq 2$,
%$$\sup_{(t,x) \in [0,T] \times \bR}E[|u(t,x)|^p]<\infty.$$
}
\end{remark}

\section*{Acknowledgement}

The authors would like to thank Robert Dalang for his valuable comments, which allowed us to improve condition (Q) and so prove the 
main result of the present note.

\end{document}